\documentclass[a4paper,11pt]{amsart}
\AtBeginDvi{}
\usepackage{amsmath,amssymb,amsthm,extarrows,mathrsfs}
\usepackage{url,here}
\usepackage[dvipdfmx]{graphicx}
\usepackage{listings}
\usepackage[all]{xy}

\DeclareMathOperator{\kak}{Ker}

\DeclareMathOperator{\Hom}{Hom}

\DeclareMathOperator{\spec}{Spec}
\DeclareMathOperator{\proj}{Proj}

\newcommand{\hifunc}{\mathcal{H}\mathrm{ilb}_{n}^{P}}
\newcommand{\hihilb}{\mathrm{Hilb}_{n}^{P}}

\newcommand{\hgrfunc}{\mathcal{G}\mathrm{r\text{\rm \"o}b}_{\prec}^{J}}
\newcommand{\hgrhilb}{\mathrm{Gr\text{\rm \"o}b}_{\prec}^{J}}

\newcommand{\bbfunc}{\mathcal{BB}^{J}_{\omega}}
\newcommand{\bbhilb}{\mathrm{BB}^{J}_{\omega}}

\newcommand{\grass}{\mathrm{G}\left(\binom{n+r}{r} - P(r), S_r\right)}

\DeclareMathOperator{\mpn}{\mathcal{M}_{P,n}}
\DeclareMathOperator{\ipn}{\mathbf{I}_{P,n}}
\DeclareMathOperator{\gm}{\mathbb{G}_m}

\DeclareMathOperator{\cd}{\mathscr{C}(\Delta)}

\DeclareMathOperator{\alg}{\mathchar`-Alg}

\DeclareMathOperator{\cset}{Set}

\DeclareMathOperator{\ini}{in}

\newcommand{\gr}{Gr\"obner}
\newcommand{\plu}{Pl\"ucker}
\newcommand{\bb}{Bia{\l}ynicki-Birula}

\newcommand{\mon}{\mathbb{N}^{n+1}}
\newcommand{\prn}{\mathbb{P}^n_k}
\newcommand{\prnn}{\mathbb{P}^N_k}

\theoremstyle{plain}
\newtheorem{theorem}{Theorem}[section]
\newtheorem{proposition}{Proposition}[section]
\newtheorem{corollary}{Corollary}[section]
\newtheorem{lemma}{Lemma}[section]

\theoremstyle{definition}
\newtheorem{definition}{Definition}[section]
\newtheorem{example}{Example}[section]

\title[Computable BB decomposition of the Hilbert scheme]{Computable Bia{\l}yniski-Birula decomposition of the Hilbert scheme}
\author{Yuta Kambe}
\date{}
\address{Graduate School of Science and Engineering, Saitama University, Japan}
\email{y.kambe.021@ms.saitama-u.ac.jp}
\thanks{This work was supported by JSPS KAKENHI [Grant Number 18J12368]}

\begin{document}
\maketitle

\begin{abstract}
We call the scheme parameterizing homogeneous ideals with fixed initial ideal the \gr\ scheme. We introduce a \bb\ decomposition of the Hilbert scheme $\hihilb$ for any Hilbert polynomial $P$ such that the cells are the \gr\ schemes in set-theoretically. Then we obtain a computable homology formula for smooth Hilbert schemes. As a corollary of our argument, we show that the \gr\ scheme for a monomial ideal defining a smooth point in the Hilbert scheme is smooth.
\end{abstract}

\section{Introduction}
Let $k$ be a field. Our main target is the Hilbert scheme $\hihilb$ parameterizing closed subschemes in $\prn$ with Hilbert polynomial $P$. This paper is organized as follows.

Let $\prec$ be an arbitrary monomial order on $S = k[x_0, \ldots, x_n]$. We see a set-theoretically decomposition of the Hilbert scheme $\hihilb$ into the loci of homogeneous ideals with fixed initial ideal (Proposition \ref{prop.decompHilb}). We denote by $\hgrhilb$ such a locus for a monomial ideal $J$ and we call $\hgrhilb$ the \textit{\gr\ scheme}. Namely, in set-theoretically,
\[ \hgrhilb = \{ I \subset S \mid \text{$I$ is a homogeneous ideal and}\ \ini_{\prec} I = J \}. \]
We denote by $\mpn$ the set of monomial ideals appearing in the decomposition:
\[ \hihilb = \coprod_{J \in \mpn} \hgrhilb. \eqno (\ast)\]
This decomposition comes from the closed embedding into the Grassmannian. We show that this decomposition is computable (Proposition \ref{prop.trunc_of_sat}).

We see examples of decompositions of smooth Hilbert schemes (Example \ref{ex.betti_of_pts}, Example \ref{ex.m+1}, Example \ref{ex.2t+1}). Those examples give us symmetrical numbers. In fact, the numbers are the Betti numbers of the Hilbert scheme. A purpose of this paper is to explain this phenomenon.

\begin{table}[H]
\begin{tabular}{|c||ccccccccccc|} \hline
 $d \setminus m$ & 0 & 1 & 2 & 3 & 4 & 5 & 6 & 7 & 8 & 9 & 10\\ \hline \hline
1 & 1 & 1 & 1 & & & & & & & & \\ \hline
2& 1& 2& 3& 2& 1& & & & & & \\ \hline
3& 1& 2& 5& 6& 5& 2& 1& & & & \\ \hline
4& 1& 2& 6& 10& 13& 10& 6& 2& 1& & \\ \hline
5& 1& 2& 6& 12& 21& 24& 21& 12& 6& 2& 1\\ \hline
\end{tabular}
\caption{The numbers of $J \in \mathcal{M}_{d,2}$ such that $\hgrhilb \cong \mathbb{A}^m_k$ \label{tab.int_betti_num}}
\end{table}

We show that the decomposition $(\ast)$ is the set-theoretically \bb\ decomposition \cite{BB73,BB76} with respect to a $\gm$-action compatible with $\prec$. Therefore we obtain a computable homology formula (Corollary \ref{cor.homology}). As a corollary of our argument, we show that the \gr\ scheme $\hgrhilb$ for $J \in \mpn$ is smooth if $\hihilb$ is smooth at $\proj S/J \in \hihilb(k)$ (Corollary \ref{cor.smoothness}). The same statement is known for segment monomial ideals which are monomial ideals with a combinatorial condition since the \gr\ scheme for a segment monomial ideal is an open subscheme of $\hihilb$ as the marked family \cite{CLMR11}. A remarkable point is that our theorem is provided in the context of geometries.

\begin{theorem}\label{thm.introsmooth}
If $\hihilb$ is smooth at $\proj S/J \in \hihilb(k)$ for $J \in \mpn$, then the \gr\ scheme $\hgrhilb$ is isomorphic to an affine space.
\end{theorem}

\begin{theorem}
Assume that $k$ is algebraically closed, $\mathrm{char}\ k = 0$ and the Hilbert scheme $\hihilb$ is smooth. Denote by $p(J)$ the dimension of the Zariski tangent space on $\hgrhilb$ at $J$. Then we have the following formula about the homology of $\hihilb$: 
\[ H_m(\hihilb, \mathbb{Z}) \cong \bigoplus_{J \in \mpn} H_{m-2p(J)}(\{J\}, \mathbb{Z}) \cong \mathop{\bigoplus_{J \in \mpn}}_{2p(J) = m} \mathbb{Z} \]
for any integer $m$ with $0 \leq m \leq 2 \dim \hihilb$. 
\end{theorem}

Sernesi construct a singular point of $\mathrm{Hilb}^{2t+2}_{3}$ defined by a monomial ideal \cite{Ser06}. As an example of the decomposition of a non-smooth Hilbert scheme, we compute the decomposition of the Hilbert scheme $\mathrm{Hilb}^{2t+2}_{3}$ with respect to $\prec_{lex}$ and $\prec_{rvlex}$ (Example \ref{ex.decomp_2t+2}). Then, by applying Theorem \ref{thm.introsmooth}, we find $18$ singular points of $\mathrm{Hilb}^{2t+2}_{3}$ defined by monomial ideals.\\

We note background. The \gr\ scheme, or also called the \textit{\gr\ stratum}, is introduced in \cite{NS00}. The \gr\ scheme is an affine scheme of finite type over $k$ and has computable defining equations \cite{Rob09, RT10, Led11}. Moreover, the \gr\ scheme $\hgrhilb$ is isomorphic to an affine space if $\hgrhilb$ is smooth at $J \in \hgrhilb(k)$ \cite[Corollary 3.7]{Rob09}, \cite[Corollary 3.6]{RT10} (see also \cite{FR09}). Precisely we define the \gr\ scheme $\hgrhilb$ as the scheme representing the following \textit{\gr\ functor}:
\[ \hgrfunc : (k\alg) \rightarrow (\cset) \]
\[ A \mapsto \left\{ G \subset A[x_0,\ldots, x_n] \left| \begin{aligned} &\text{$G$ is a reduced \gr\ basis consisting of}\\ &\text{homogeneous polynomials, $\ini_{\prec} \langle G \rangle = J \otimes_k A$} \end{aligned} \right. \right\}.\]
See \cite{Wib07} for definition of reduced \gr\ bases that coefficients are in a ring. Sometime a property of \gr\ schemes is not compatible with the Hilbert scheme as schemes since a general ideal in $A[x_0, \ldots, x_n]$ may not have a reduced \gr\ basis. However, the \gr\ scheme $\hgrhilb$ is a locally closed subscheme of $\hihilb$ if $J \in \mpn$ \cite[Theorem 5.3]{LR16}.

The \textit{\gr\ deformation} or the \textit{\gr\ degeneration}, given in \cite{Bay82}, certainly exists for any element $G \in \hgrfunc(A)$ in case $A$ is an arbitrary commutative ring. That is the flat family of closed subschemes $\{ Y_t \}$ in $\prn \times_k \spec A$ over $\mathbb{A}^1_k \times_k \spec A$ such that letting $I = \langle G \rangle$, general fibers are isomorphic to $\proj A[x_0,\ldots, x_n]/I$ and the special fiber at $\{0\} \times_k \spec A$ is isomorphic to $\proj A[x_0, \ldots, x_n]/(\ini_{\prec} I)$. The key point is that the \gr\ degeneration is provided as the orbit of $I$ with the limit $\ini_{\prec} I$ by a $\gm$-action on the polynomial ring compatible with $\prec$.

On the other hand, \bb\ introduces significant loci in a scheme $X$ with a $\gm$-action. Nowadays these are called \textit{\bb\ cells} or \textit{\bb\ schemes}. For simplicity, we assume that $X$ is smooth projective over an algebraically closed field $k$ and has finite fixed points $X^{\gm} = \{a_1, \ldots, a_r \}$. Therefore any orbit of $x \in X$ has a limit $\lim_{t \rightarrow 0} t \cdot x \in X^{\gm}$ by the extension of the orbit morphism $\gm \rightarrow X ; t \mapsto t \cdot x$ to $t = 0$. The \bb\ scheme $X_i^{+}$ is defined as
\[ X_i^{+} = \{ x \in X \mid \lim_{t \rightarrow 0} t \cdot x = a_i \}. \]
The \bb's theorem gives us that any $X_i^{+}$ is isomorphic to an affine space and $\{ X_i^{+} \}$ gives a cell decomposition of $X$ \cite{BB73,BB76}. Recently, for an arbitrary $X$ locally of finite type, the BB schemes have been defined and investigated in \cite{Dri13, JS18}. Thanks to \cite{Dri13, JS18}, we combine Bayer's degeneration and \bb's idea on the Hilbert scheme $\hihilb$.

\begin{theorem}
{\rm (Proposition \ref{prop.fixed_monomial}, Theorem \ref{thm.bb_is_grob}).} There exists a $\gm$-action on $\hihilb$ such that the scheme of fixed points is $(\hihilb)^{\gm} = \{ \proj S/J \mid J \in \mpn \}$ in set-theoretically and the BB scheme $(\hihilb)_{J}^{+}$for $\proj S/J$ is $\hgrhilb$ in set-theoretically.
\end{theorem}

\begin{theorem}
{\rm (Theorem \ref{thm.smoothness_on_gr}).} If the BB scheme $(\hihilb)_J^{+}$ is smooth at $\proj S/J \in (\hihilb)_J^{+}(k)$, then the \gr\ scheme $\hgrhilb$ is isomorphic to an affine space.
\end{theorem}

\begin{theorem}
{\rm (Proposition \ref{prop.local_smoothness_of_BB}).} If $\hihilb$ is smooth at $\proj S/J \in \hihilb(k)$, then the BB scheme $(\hihilb)^{+}_{J}$is smooth at $\proj S/J \in (\hihilb)^{+}_{J}(k)$.
\end{theorem}

Such a combination already has been investigated for the Hilbert scheme $H^d_n = H^d(\mathbb{A}^n_k)$ of $d$ points in the affine space $\mathbb{A}^n_k$, but not been investigated for an arbitrary $\hihilb$. Let us recall results on $H^d_n$. The BB scheme in $H^d_n$ with respect to weights on coordinates of $\mathbb{A}^n_k$ is described in \cite{EL12} as the intersection of schemes determined by an argument of initial ideals. Moreover, \cite{EL12} shows that the BB schemes included in a fiber of the Hilbert-Chow morphism corresponding to a coordinate with negative weight. These results seem to be similar and related to our construction of the decomposition $(\ast)$. \cite{Jel17} deals with the obstruction theory on the BB schemes in $H^d_n$ and show that Murphy's Law holds on the BB schemes. However, in this paper, we do not deal with Murphy's Law and obstructions of the BB schemes in $\hihilb$.

\section{Preliminaries and Notation}
\begin{itemize}
\item Let $k$ be a field and $S=k[x] = k[x_0, \ldots, x_n]$ the polynomial ring over $k$ in $(n+1)$ variables. We always fix a monomial order $\prec$ on $S$. We consider the ordinal degrees of polynomials in $S$. For a subset $A \subset S$, we denote by $A_r$ the homogeneous elements of $A$ with degree $r$ and denote by $\langle A \rangle$ the ideal generated by $A$ in $S$.
\item For $\alpha = (\alpha_0, \ldots, \alpha_n) \in \mon$, let $x^{\alpha} = x_0^{\alpha_0} \cdots x_n^{\alpha_n}$. Using this notation, we regard $\mon$ as the set of monomials in $(n+1)$ variables. The degree of $\alpha$ is $|\alpha| = \alpha_0 + \cdots + \alpha_n$. For a subset $A \subset \mon$, let $A_r = \{ \alpha \in A \mid |\alpha| = r\}$. Let $e_i$ be the $i$-th canonical vector.
\item For $k$-schemes $X$ and $Y$, let $X(Y) = \Hom_k (Y, X)$. If $Y = \spec A$, we denote it by $X(A)$ instead.
\end{itemize}

The \textit{Hilbert scheme} $\hihilb$ is the scheme representing the following \textit{Hilbert functor}:
\[ \hifunc : (k\alg) \rightarrow (\cset) \]
\[  A \mapsto \left\{ Y \subset \mathbb{P}^{n}_{A} \left| \begin{aligned} &\text{$Y$ is a closed subscheme in $\mathbb{P}^n_{A}$ flat over $\spec A$,}\\
&\text{the Hilbert polynomials of all fibers on closed}\\
&\text{points of $\spec A$ are $P$} \end{aligned} \right. \right\}. \]

There exists a canonical morphism $\hgrhilb \rightarrow \hihilb$ induced by the natural transformation
\[ \hgrfunc \rightarrow \hifunc \]
\[ G \mapsto \proj A[x]/ \langle G \rangle. \]
If we denote a morphism $\hgrhilb \rightarrow \hihilb$, we always mean this morphism.

\section{Computable decomposition of the Hilbert scheme into the \gr\ schemes}

We recall the embedding of the Hilbert scheme into the Grassmannian. See \cite{HS04}, \cite{MacD07} or other references about Grothendieck's construction of Hilbert schemes.

\begin{proposition}\label{prop.Gotzmann}
{\rm \cite[Corollary B.5.1]{Vas98}} Let $P$ be a Hilbert polynomial of a closed subscheme of a projective space. Then there exist integers $a_1 \geq a_2 \geq \cdots \geq a_r \geq 0$ such that
\[ P(t) = \sum_i^r \binom{t+a_i-i+a}{a_i}. \]
The number $r$ is called the \textit{Gotzmann number} of $P$.
\end{proposition}
We call $I_{\geq r} = \oplus_{s \geq r} I_s$ the \textit{$r$-truncation} of $I$. We say $I$ is \textit{$r$-truncated} if $I_{\geq r} = I$. Moreover, we denote by $I_{\mathrm{sat}}$ the \textit{saturation} of $I$ by $\langle x_0, \ldots, x_n \rangle$. We say $I$ is \textit{saturated} if $I_{\mathrm{sat}} = I$. If $I$ is a homogeneous saturated ideal in $S$ such that the Hilbert polynomial of $S/I$ is $P$, then $\dim_k (S/I)_{s} = P(s)$ for any $s \geq r$.

We have a closed embedding
\[ \hihilb \hookrightarrow \grass, \]
where $\grass$ is the Grassmannian. Then we can describe
\[ \hihilb(A) \cong \left\{ I \subset A[x] \left| \begin{aligned}
&\text{$I$ is the $r$-truncation of a saturated ideal}\\
&\text{defining an element of $\hifunc(A)$} \end{aligned} \right. \right\}. \]
For short, we denote by $\ipn$ the above condition: $\ipn \equiv$ $I$ is the $r$-truncation of a saturated ideal defining an element of $\hifunc$. The image of $\hifunc(k)$ in $\Hom_k(\spec k, \grass)$ is the set of subspaces $V \subset S_r$ such that $\dim_k V = \binom{n+r}{r} -P(r)$ and $\dim_k S_1\cdot V = \binom{n+r+1}{r+1} -P(r+1)$. Therefore the condition $\ipn$ is equivalent to the following condition: $I$ is generated by $I_r$, $\dim_k I_{r} = \binom{n+r}{r} - P(r)$ and $\dim_k I_{r+1} = \binom{n+r+1}{r+1} - P(r+1)$.

We introduce a decomposition of the Hilbert scheme into the \gr\ schemes.

\begin{lemma}\label{lem.initial_sat_ipn}
A homogeneous ideal $I$ in $S$ satisfies $\ipn$ if and only if the initial ideal $J= \ini_{\prec} I$ satisfies $\ipn$.
\end{lemma}
\begin{proof}
Assume that $I$ satisfies $\ipn$. For any $s \geq r$, we have $\dim_k (S/J)_{s} = \dim_k (S/I)_{s} = P(s)$. The Hilbert polynomial of $\proj S/J$ in $\mathbb{P}^n_k$ is also $P$, so $\dim_k (S/J_{\mathrm{sat}})_s = P(s)$. Thus $(J_{\mathrm{sat}})_{\geq r} = J_{\geq r}= \ini_{\prec}( I_{\geq r}) = J$. Conversely, assume that $J$ satisfies $\ipn$. Then there exists a saturated monomial ideal $J'$ such that $J = J'_{\geq r}$ and $\proj S/J'  \in \hifunc(k)$. Put $I' = I_{\mathrm{sat}}$. Then for any $s \geq r$, we have $\dim_k (S/I)_{s} = \dim_k (S/J)_{s} = \dim_k (S/J')_{s} = \dim_k (S/I')_{s} = P(s)$. Therefore we obtain $I = I'_{\geq r}$.
\end{proof}

\begin{proposition}\label{prop.locally_closed}
{\rm(\cite[Theorem 5.3]{LR16})} The morphism $\hgrhilb \rightarrow \hihilb$ is a locally closed immersion if $J$ satisfies $\ipn$.
\end{proposition}

\begin{proposition}\label{prop.decompHilb}
Let $\mpn$ be the set of monomial ideals satisfying $\ipn$. Then for any field extension $k \subset K$,
\[ \hifunc(K) = \coprod_{J \in \mpn}  \hgrfunc(K). \]
Namely, in set-theoretically,
\[ \hihilb = \coprod_{J \in \mpn} \hgrhilb. \]
\end{proposition}

\begin{proof}
This is easy from
\[ \hgrfunc(K) \cong \{ I \subset K[x] \mid \text{$I$ is a homogeneous ideal with $\ini_{\prec} I = J \otimes_k K$} \}\]
and Lemma \ref{lem.initial_sat_ipn}.
\end{proof}

Let us see that the set $\mpn$ is computable.

\begin{definition}\label{def.standard_set}
{\rm (\cite{Led11})} A subset $\Delta$ in $\mon$ is a \textit{standard set} if $\alpha + \beta \in \Delta$ implies $\alpha, \beta \in \Delta$ for any $\alpha , \beta \in \mon$. There is a one-to-one correspondence between the set of standard sets in $\mon$ and the set of monomial ideals in $S$ given by $\Delta \mapsto J_{\Delta} = \langle x^{\alpha} \mid \alpha \in \mon \setminus \Delta \rangle$. For a standard set $\Delta$, we define the set of \textit{corners}
\[ \cd = \{\alpha \in \mon \setminus \Delta \mid \forall\,i \ \alpha - e_i \not\in \mon \setminus \Delta \}. \]
The set of corners $\cd$ corresponds to the minimal generators of $J_{\Delta}$.
\end{definition}

\begin{proposition}\label{prop.trunc_of_sat}
Let $\Delta$ be a standard set. The monomial ideal $J_{\Delta}$ is an element of $\mpn$ if and only if the set of corners $\cd$ satisfies
\begin{itemize}
\item[(1)] $\cd \subset (\mon)_{r}$,
\item[(2)] $\#(\cd) = \binom{n+r}{r} - P(r)$,
\item[(3)] $\#(\{ \alpha + e_i \mid \alpha \in \cd, i=0, \ldots, n \} ) = \binom{n+r+1}{r+1} - P(r+1)$.
\end{itemize}

\end{proposition}

Therefore the set-theoretical decomposition
\[ \hihilb = \coprod_{J \in \mpn} \hgrhilb \]
is computable.

\begin{example}\label{ex.betti_of_pts}
We compute an example of the decomposition in Proposition \ref{prop.decompHilb}. We consider the Hilbert scheme of $d$ points in $\mathbb{P}^2_k$. The Hilbert scheme $\mathrm{Hilb}^d_{2}$ is smooth and its dimension is $2d$ \cite{Har10}. Using the same argument of \cite{Har10} (i.e. using an obstruction theory on the \gr\ scheme), we obtain that the \gr\ scheme $\hgrhilb$ is isomorphic to an affine space $\mathbb{A}^m_k$ for any $J \in \mathcal{M}_{d,2}$. We make a table of the numbers of $J \in \mathcal{M}_{d,2}$ such that $\hgrhilb \cong \mathbb{A}^m_k$. In fact, the numbers are the Betti numbers of the Hilbert schemes \cite{ES87}.

\end{example}

\begin{example}\label{ex.m+1}
Let us consider the case $P = t+1$, $n=3$. Then the Hilbert scheme $\mathrm{Hilb}^{t+1}_{3}$ parameterizes lines in $\mathbb{P}^3_k$ and isomorphic to the Grassmannian $\mathrm{G}(1,3)$. The numbers of $J \in \mathcal{M}_{t+1,3}$ are on Table \ref{tab.betti_num_m+1}. The Betti numbers of $\mathrm{G}(1,3)$ is computed by determining Schubert cycles in $\mathrm{G}(1,3)$ \cite{Ehr34}. The numbers on Table \ref{tab.betti_num_m+1} are just the Betti numbers of $\mathrm{G}(1,3)$.
\end{example}

\begin{table}[h]
\begin{tabular}{|c||ccccccccccc|} \hline
 $d \setminus m$ & 0 & 1 & 2 & 3 & 4 & 5 & 6 & 7 & 8 & 9 & 10\\ \hline \hline
1 & 1 & 1 & 1 & & & & & & & & \\ \hline
2& 1& 2& 3& 2& 1& & & & & & \\ \hline
3& 1& 2& 5& 6& 5& 2& 1& & & & \\ \hline
4& 1& 2& 6& 10& 13& 10& 6& 2& 1& & \\ \hline
5& 1& 2& 6& 12& 21& 24& 21& 12& 6& 2& 1\\ \hline
\end{tabular}
\caption{The numbers of $J \in \mathcal{M}_{d,2}$ such that $\hgrhilb \cong \mathbb{A}^m_k$ \label{tab.betti_num}}
\end{table}

\begin{table}[h]
\begin{tabular}{|c||ccccccccccc|} \hline
 $m$ & 0 & 1 & 2 & 3 & 4 & 5 & 6 & 7 & 8 & 9 & 10\\ \hline \hline
& 1 & 1 & 2 & 1& 1& & & & & & \\ \hline
\end{tabular}
\caption{The numbers of $J \in \mathcal{M}_{t+1,3}$ such that $\hgrhilb \cong \mathbb{A}^m_k$ \label{tab.betti_num_m+1}}
\end{table}

The Betti numbers of $\mathrm{Hilb}^d_2$ is determined in \cite{ES87} by computing a \textit{\bb\ decomposition} of $\mathrm{Hilb}^{d}_2$ \cite{BB73,BB76}. The next purpose is to show that the decomposition
\[ \hihilb = \coprod_{J \in \mpn} \hgrhilb \]
is just the \bb\ decomposition with respect to a $\gm$-action on $\hihilb$.

\section{$\gm$-action on the Hilbert scheme compatible with a monomial order}

\begin{proposition}\label{prop.weight}
{\rm (\cite[Proposition 1.8]{Bay82})} Let $\prec$ be a monomial order on $S$, and let $A$ be a finite subset of $\mon$. Then there exists a vector $\omega \in \mon$ such that for any $\alpha, \beta \in A$, $\alpha \prec \beta$ if and only if $\omega \cdot \alpha < \omega \cdot \beta$. Here $\omega \cdot \alpha$ is the ordinary inner product $\omega_0 \alpha_0 + \cdots \omega_n \alpha_n$.
\end{proposition}

We fix a vector $\omega \in \mon$ given by Proposition \ref{prop.weight} for the finite subset $(\mon)_r$. This vector $\omega$ implies a $\mathbb{G}_m$-action on the \gr\ scheme $\hgrhilb$ for each $J \in \mpn$ as follows.

\begin{proposition}\label{prop.coordinate_of_grobner_scheme}
{\rm (\cite{RT10, Led11})} The \gr\ scheme $\hgrhilb$ for $J = J_{\Delta} \in \mpn$ is isomorphic to a closed subscheme of $\spec k[T_{\alpha, \beta} \mid \alpha\in \cd, \beta \in \Delta_{r}, \alpha \succ \beta]$. We define a grading on $R =k[T_{\alpha, \beta} \mid \alpha \in \cd, \beta \in \Delta_{r}, \alpha \succ \beta]$ such that $\mathrm{deg} (T_{\alpha, \beta}) = \omega \cdot \alpha - \omega \cdot \beta$ and attach a $\gm$-action on $R$ from this grading. Then the \gr\ scheme $\hgrhilb$ is $\gm$-invariant in $\spec R$.
\end{proposition}
The vector $\omega$ defines $\gm$-actions on $S$ and on $S_r$ as a negative grading $t \cdot x^{\alpha} = t^{-\omega \cdot \alpha} x^{\alpha}$. Therefore there exist $\mathbb{G}_m$-actions on the Hilbert scheme $\hihilb$ and on the Grassmannian $\grass$ respectively such that $\hihilb \hookrightarrow \grass$ is $\gm$-equivariant. Moreover, we also obtain a $\gm$-action on the projective space $\mathbf{P} =\mathbb{P}(\wedge^{\binom{n+r}{r}-P(r)} S_r)$ such that the \plu\ embedding $\grass \hookrightarrow \mathbf{P}$ is $\gm$-equivariant.

\begin{proposition}\label{prop.gm-equiv}
If $J \in \mpn$, then the morphism $\hgrhilb \rightarrow \hihilb$ is a $\mathbb{G}_m$-equivariant morphism. 
\end{proposition}
\begin{proof}
For each reduced \gr\ basis 
\[ G = \left\{ \left. g_{\alpha} = x^{\alpha} - \sum_{\beta \in \Delta} a_{\alpha, \beta} x^{\beta} \right| \alpha \in \cd\right\} \in \hgrfunc(A), \]
we have
\[t \cdot g_{\alpha} = t^{- \alpha \cdot \omega} x^{\alpha} - \sum_{\beta \in \Delta} t^{-\beta \cdot \omega} a_{\alpha, \beta} x^{\beta} \quad (t \in A^{\times})\]
under the $\mathbb{G}_m$-action on $A[x]$. Let $I$ be the ideal generated by $G$, and let $Y = \proj A[x]/I$. Then $t \cdot I = \{t\cdot f \mid f\in I \}$ is generated by the set $\{ x^{\alpha} - \sum_{\beta \in \Delta} t^{\alpha \cdot \omega - \beta \cdot \omega} a_{\alpha, \beta} x^{\beta} \mid \alpha \in \cd \}$. We have $(t \cdot I)_{\geq r} = t \cdot I_{\geq r}$ for any integer $r \geq 0$. Since $I = (I_{\mathrm{sat}})_{\geq r}$, we have
\[ t \cdot Y = \proj A[x]/(t \cdot I_{\mathrm{sat}}) = \proj A[x]/(t \cdot I_{\mathrm{sat}})_{\geq r} = \proj A[x]/(t \cdot I). \]
Thus the morphism $\hgrhilb \rightarrow \hihilb$ is a $\mathbb{G}_m$-equivariant morphism.
\end{proof}

From now on, we always attach the $\gm$-action on $\hihilb$ introduced in the above for given monomial order $\prec$.

\section{\bb\ schemes in the Hilbert scheme}

Let $X$ be a scheme locally of finite type over $k$ with a $\gm$-action. For any $k$-algebra $A$, we attach a $\gm$-action on $\spec A$ as the projection $\gm \times_k \spec A \rightarrow \spec A$. We also attach the trivial $\mathbb{G}_m$-action on $\mathbb{A}^1_k \times_k \spec A$ induced from the canonical $\gm$-action on $\mathbb{A}^1_k$.

The \textit{scheme of fixed points} \cite{Fog73} is defined as the subscheme $X^{\gm}$ such that for any $k$-algbera $A$,
\[ X^{\gm}(A) = \{ \varphi \in X(A) \mid \text{$\varphi$ is $\gm$-equivariant}\}. \]
The scheme of fixed points exists and it is a closed subscheme of $X$ \cite[Proposition1.2.2]{Dri13}.

We define the \textit{scheme of attractors} in $X$ as the scheme $X^{+}$ such that for any $k$-algebra $A$,
\[ X^{+}(A) \cong \{ \varphi : \mathbb{A}^1_k \times_k \spec A \rightarrow X \mid \text{$\varphi$ is $\gm$-equivariant} \}. \]
The scheme of attractors exists and it is locally of finite type over $k$ \cite[Corollary 1.4.3]{Dri13}, \cite[Theorem 6.17]{JS18}.

\begin{proposition}\label{prop.fixed_monomial}
The scheme of fixed points of the Hilbert scheme $\hihilb$ satisfies $(\hihilb)^{\gm}(K) = \{ \proj K[x]/(J \otimes_k K) \mid J \in \mpn \}$ for any field extension $k \subset K$. In particular, we have $X^{\gm} = \{ \proj S/J \mid J \in \mpn \}$ in set-theoretically.
\end{proposition}

\begin{proof}
Let $I$ be a homogeneous ideal that is the $r$-truncation of a saturated ideal defining an element of $\hifunc(K)$. Then there exists a monomial ideal $J \in \mpn$ such that $\ini_{\prec} I = J \otimes_k K$. Let $\Delta$ be the standard set attached to $J$. The reduced \gr\ basis of $t \cdot I$ $(t \in K \setminus \{0\})$ is in the following form:
\[ G = \left\{ \left. x^{\alpha} - \sum_{\beta \in \Delta} t^{\omega \cdot \alpha - \omega \cdot \beta} a_{\alpha, \beta} x^{\beta} \right| \alpha \in \cd \right\}. \]
Therefore $\proj K[x]/I$ is a fixed point if and only if $a_{\alpha, \beta} = 0$ for any $\alpha \in \cd$ and $\beta \in \Delta$ since $\hgrhilb \rightarrow \hihilb$ is $\gm$-equivariant (Proposition \ref{prop.gm-equiv}).
\end{proof}

We obtain canonical maps by taking restrictions to $1$:
\[ i_{X}: X^{+}(A) \rightarrow X(A) \]
\[ \varphi \mapsto \varphi_{| \{1\} \times_k \spec A }. \]
If $X$ is separated, then this map is an injection for each $A$ \cite[Proposition 1.4.11]{Dri13}.
We also obtain maps by taking restrictions to $0$:
\[ \pi_X : X^{+}(A) \rightarrow X^{\gm}(A) \]
\[ \varphi \mapsto \varphi_{| \{0\} \times_k \spec A}. \]
This morphism $\pi_X : X^{+} \rightarrow X^{\gm}$ is $\gm$-equivariant and affine of finite type \cite[Theorem 6.17]{JS18}.

We describe the connected components of $X^{\gm}$ by $F_1, \ldots, F_r$. The \textit{\bb\ schemes} are defined as the preimages of components under $\pi_X$. More precisely, the \bb\ scheme $X_i^{+}$ is the subscheme of $X^{+}$ such that
\[ X_i^{+}(A) = \{ \varphi \in X^{+}(A) \mid \pi_{X}(\varphi) \in F_i(A) \}. \]
The right side set is the sections of the \textit{\bb\ functor}. For short, we call these by BB scheme and BB functor respectively.

If $X$ is the Hilbert scheme $\hihilb$, then each connected component of $X^{\gm}$ is a point corresponding to a monomial ideal in $\mpn$. Thus we denote the BB functor and the BB scheme for $J \in \mpn$ by $\bbfunc$ and $\bbhilb$ respectively.

\begin{theorem}\label{thm.bb_is_grob}
Let $J$ be an element of $\mpn$. Then for any field extension $k \subset K$, we have
\[ \bbfunc(K) = \hgrfunc(K) \]
in the Hilbert functor $\hifunc(K)$. Namely, $\bbhilb = \hgrhilb$ in set-theoretically.
\end{theorem}

\begin{proof}
Taking \gr\ degenerations \cite[Proposition 2.12]{Bay82}, $\hgrhilb$ is a subscheme of $\bbhilb$. Then we obtain $\hgrfunc(K) \subset \bbfunc(K)$. Conversely, for any $\varphi \in \bbfunc(K)$, put $Y = \varphi_{|\{1\}} \in \hifunc(K)$ and assume that $Y \in \mathcal{G}\mathrm{r\text{\"o}b}_{\prec,h}^{J'}(K)$ with $J' \in \mpn$. Then taking the \gr\ degeneration of $Y$, there exists a $\gm$-equivariant morphism $\psi : \mathbb{A}^1_K \rightarrow \hihilb$ such that $\psi_{| \{1\}} = Y$ and $\psi_{| \{0\}} = \proj K[x]/J'$. Since $({\hihilb})^{+} \rightarrow \hihilb$ is monomorphism, we obtain $\varphi = \psi$. Hence $J=J'$.
\end{proof}

\section{Smoothness}

Let $X$ still be a scheme locally of finite type over $k$ with a $\gm$-action. We recall the following \bb's result.

\begin{theorem}\label{thm.bb_decomp}
{\rm (\cite{BB73,BB76}, see also \cite{Dri13,JS18})} Let $X$ be a smooth projective scheme over an algebraically closed field $k$ with a $\gm$-action. We assume that $X^{\gm}$ is $0$-dimensional. Then there exist closed subschemes $Z_0 \supset \cdots \supset Z_q$ such that
\begin{itemize}
\item $Z_0 = X$ and $Z_q = \emptyset$,
\item each $Z_i \setminus Z_{i+1}$ is a BB scheme in $X$,
\item any BB scheme is isomorphic to an affine space over $k$.
\end{itemize}
Therefore $X$ has a cell decomposition.
\end{theorem}

Here we say that a sequence of closed subschemes $Z_0 \supset \cdots \supset Z_q$ is a \textit{cell decomposition} of $X$ \cite{Ful98} if
\begin{itemize}
\item $Z_0 = X$ and $Z_q = \emptyset$,
\item each $Z_i \setminus Z_{i+1}$ is the disjoint sum of schemes isomorphic to affine spaces.
\end{itemize}

The above \bb's result is generalized as follows.

\begin{theorem}\label{thm.smoothness_on_BB}
{\rm (\cite[Corollary 7.3]{JS18}). } Suppose that $X$ is smooth over $k$. Then $\pi_X : X^{+} \rightarrow X^{\gm}$ is an affine fiber bundle. Moreover, both $X^{\gm}$ and $X^{+}$ are smooth.
\end{theorem}

The next purpose is to apply the above theorems to our Hilbert scheme and \gr\ schemes.

\begin{lemma}\label{lem.dimensions}
Let $A$ and $B$ be Noetherian local $k$-algebras with residue field $k$. Assume that $B$ is regular and there exists a $k$-morphism $\varphi : B \rightarrow A$ such that $\varphi$ induces bijections $l.\Hom_k (A, K) \rightarrow l.\Hom_k (B, K)$ for any filed extension $k \subset K$. Here we denote by $l.\Hom_k$ the local ring $k$-morphisms. Then $\dim A \geq \dim B$.  
\end{lemma}

\begin{proof}
Let $\hat{A}$ and $\hat{B}$ be the completion of $A$ and $B$ respectively. Then $\hat{B} \cong k[[z_1, \ldots, z_m]]$ for some $m \in \mathbb{N}$ by Cohen's structure theorem. Let $K_i$ be the fraction field of $k[[z_1, \ldots, z_i]]$. There exist canonical morphisms $\psi_i : \hat{B} \rightarrow K_i$ and $\eta_i : K_{i+1} \rightarrow K_i$. Since $\varphi^{\ast} : l.\Hom_k (A, K_i) \rightarrow l.\Hom_k (B, K_i)$ is bijective, $\hat{\varphi}^{\ast} : l.\hom_k (\hat{A}, K_i) \rightarrow l.\Hom_k (\hat{B}, K_i)$ is also bijective. Then there uniquely exists a morphism $\rho_i : \hat{A} \rightarrow K_i$ such that $\psi_i  = \rho_i \circ \hat{\varphi}$. Since the diagram

\[ \xymatrix{
l.\Hom_k (\hat{A}, K_{i+1}) \ar[r] \ar[d] & l.\Hom_k(\hat{B},K_{i+1}) \ar[d] \\
l.\Hom_k (\hat{A}, K_i) \ar[r] & l.\Hom_k (\hat{B}, K_i) 
} \]
is commutative, the diagram
\[ \xymatrix{
\hat{A} \ar[r]^{\rho_{i+1}} \ar[rd]_{\rho_{i}} & K_{i+1} \ar[d]^{\eta_i} \\
 & K_{i}
} \]
is also commutative. Then $\kak \rho_i \subset \kak \rho_{i+1}$ for each $i$ with $0 \leq i \leq m$. We have $\hat{\varphi}(z_{i+1}) \in \kak \rho_{i+1} \setminus \kak \rho_i$, thus the sequence $\kak \rho_{0} \subset \kak \rho_{1} \subset \cdots \subset \kak \rho_{m}$ is a strictly ascending chain of prime ideals. Therefore $\dim A = \dim \hat{A} \geq m = \dim B$.
\end{proof}

\begin{theorem}\label{thm.smoothness_on_gr}
For any $J \in \mpn$, if $\bbhilb$ is smooth at $\proj S/J \in \bbhilb(k)$, then $\hgrhilb$ is isomorphic to an affine space.
\end{theorem}

\begin{proof}
Let $\mathcal{O}_B$ and $\mathcal{O}_{G}$ be the local rings of $\bbhilb$ at $\proj S/J \in \bbhilb(k)$ and of $\hgrhilb$ at $J \in \hgrhilb(k)$ respectively. Since the morphism $\hgrhilb \rightarrow \bbhilb$ maps $\proj S/J$ to $J$, there exists a morphism $\spec \mathcal{O}_G \rightarrow \spec \mathcal{O}_B$ that implies a bijective $l.\Hom_k (\mathcal{O}_B, K) \rightarrow l.\Hom_k (\mathcal{O}_G, K)$ for any field extension $k \subset K$ (Theorem \ref{thm.bb_is_grob}). Then $\dim \mathcal{O}_G \geq \dim \mathcal{O}_B$ by Lemma \ref{lem.dimensions}. Let $T_{G}$ be the Zariski tangent space on $\hgrhilb$ at $J$ and $T_{B}$ the Zariski tangent space on $\bbhilb$ at $\proj S/J$. We claim that the $k$-linear map $T_G \rightarrow T_B$ induced by $\hgrhilb \rightarrow \bbhilb$ is injective. Indeed, we can regard $T_G$ and $T_B$ as the subsets of $\Hom_{k} (\spec k[\varepsilon]/ \langle \varepsilon^2 \rangle, \hgrhilb)$ and $\Hom_{k}(\spec k[\varepsilon]/\langle \varepsilon^2 \rangle, \bbhilb)$ respectively \cite{Har77}, and the morphism $\hgrhilb \rightarrow \bbhilb$ is monomorphism since $J \in\mpn$. In fact, there exists a closed embedding $\hgrhilb \rightarrow T_G$ as schemes \cite{FR09,RT10}. Therefore by $\dim \mathcal{O}_G \leq \dim_k T_G \leq \dim_k T_B = \dim \mathcal{O}_B \leq \dim \mathcal{O}_G$, the closed embedding $\hgrhilb \rightarrow T_G$ is an isomorphism. 
\end{proof}

\begin{corollary}\label{cor.in_smooth_Hilb}
Assume that $\hihilb$ is smooth over $k$. Then the \gr\ scheme $\hgrhilb$ is isomorphic to an affine space for any $J \in \mpn$.
\end{corollary}

We localize the assumption of Corollary \ref{cor.in_smooth_Hilb}. Namely, we show that $\hgrhilb$ is isomorphic to an affine space if $\hihilb$ is smooth at $\proj S/J \in \hihilb(k)$.

\begin{proposition}\label{prop.open_immersion}
{\rm (\cite[Proposition 5.2]{JS18}). } Let $f : X \rightarrow Y$ be a $\gm$-equivariant morphism. If $f$ is an open immersion, then the induced morphism $f^{+} : X^{+} \rightarrow Y^{+}$ is also an open immersion.
\end{proposition}

\begin{proposition}\label{prop.local_smoothness_of_BB}
Let $x \in X^{\gm}(k)$. Assume that $\dim X^{\gm} = 0 $ and $X$ is smooth at $x$. Then the BB scheme $X^{+}_{x}$ for $x$ is smooth at $x$.  
\end{proposition}

\begin{proof}
Let $U$ be the smooth locus of $X$. Then $U$ is $\gm$-invariant, smooth and open in $X$. From Proposition \ref{prop.open_immersion}, $U^{+}$ is open in $X^{+}$. Then the BB scheme $U^{+}_{x}$ is also open in the BB scheme $X^{+}_{x}$. Therefore $X^{+}_{x}$ is smooth at $x$ by Theorem \ref{thm.smoothness_on_BB}.
\end{proof}

Therefore we obtain the following by Theorem \ref{thm.smoothness_on_gr} and Proposition\ref{prop.local_smoothness_of_BB}.

\begin{corollary}\label{cor.smoothness}
For any $J \in \mpn$, if the Hilbert scheme $\hihilb$ is smooth at $\proj S/J \in \hihilb(k)$, then the \gr\ scheme $\hgrhilb$ is isomorphic to an affine space.
\end{corollary}

The converse is not true by the following example.

\begin{example}\label{ex.decomp_2t+2}
In \cite{Ser06}, Sernesi shows that the Hilbert scheme $\mathrm{Hilb}^{2t+2}_{3}$ is singular at a monomial scheme. To find other singular points, let us compute our decomposition of $\mathrm{Hilb}^{2t+2}_{3}$ with respect to the reverse lexicographic order $\prec = \prec_{rvlex}$ on $k[x,y,z,w]$ such that $x \succ y \succ z \succ w$. Then we obtain:
\begin{itemize}
\item $\#(\mathcal{M}_{2t+2,3})= 159$.
\item The $144$ monomial ideals in $\mathcal{M}_{2t+2,3}$ define smooth \gr\ schemes. The dimensions are in the Table \ref{tab.decomp_2t+2_smooth}.
\item The following $15$ monomial ideals in $\mathcal{M}_{2t+2,3}$ define singular \gr\ schemes:
\[ \begin{aligned} &J_{1} = \langle w^3, zw^2, yw^2,yzw, y^2w, y^2z, y^3, xw^2,xyw,xyz,xy^2,x^2y \rangle,\\
&J_2 = \langle w^3,zw^2,yw^2,xw^2,xzw,xz^2,xyw,xyz,x^2w,x^2z,x^2y,x^3 \rangle,\\
&J_3 = \langle w^3,zw^2,yw^2,xw^2,xzw,xyw,xyz, xy^2,x^2w,x^2z,x^2y,x^3 \rangle,\\
&J_4 = \langle zw^2, z^2w, yzw, xw^2, xzw, xz^2, xyw, xyz, x^2w, x^2z, x^2y, x^3 \rangle,\\
&J_5 = \langle z^2w, z^3, yzw, yz^2, y^2w, y^2z, y^3, xzw, xz^2, xyz, xy^2, x^2z \rangle,\\
&J_6 = \langle z^2w, z^3, yzw, yz^2, y^2w, y^2z, y^3, xz^2, xyw, xyz, xy^2, x^2y \rangle,\\
&J_7 = \langle z^2w, z^3, yzw, yz^2, y^2z, xzw, xz^2, xyw, xyz, xy^2, x^2z, x^2y \rangle,\\
&J_8 = \langle z^2w, z^3, yzw, yz^2, y^2z, xzw, xz^2, xyz, x^2w, x^2z, x^2y, x^3 \rangle,\\
&J_9 = \langle z^2w,z^3,yz^2, xzw, xz^2, xyw, xyz, xy^2, x^2w, x^2z, x^2y, x^3 \rangle,\\
&J_{10} = \langle yw^2, yzw, y^2w, y^2z, y^3, xw^2, xzw, xyw, xyz, xy^2, x^2w, x^2y \rangle,\\
&J_{11} = \langle yw^2, yzw, y^2w, xw^2, xzw, xyw, xyz, xy^2, x^2w, x^2z, x^2y, x^3 \rangle,\\
&J_{12} = \langle yzw, yz^2, y^2w, y^2z, y^3, xzw, xz^2, xyw, xyz, xy^2, x^2z, x^2y \rangle,\\
&J_{13} = \langle yzw, yz^2, y^2w, y^2z, y^3, xyw, xyz, xy^2, x^2w, x^2z, x^2y, x^3 \rangle,\\
&J_{14} = \langle yzw, yz^2, y^2z, xzw, xz^2, xyw, xyz, xy^2, x^2w, x^2z, x^2y, x^3 \rangle,\\
&J_{15} = \langle y^2w, y^2z, y^3, xzw, xz^2, xyw, xyz, xy^2, x^2w, x^2z, x^2y, x^3 \rangle.
\end{aligned} \]
\end{itemize}
Therefore $\mathrm{Hilb}^{2t+2}_{3}$ includes the $15$ singular points defined by the above $15$ monomial ideals.\\

Let us change the monomial order to the lexicographic order $\prec = \prec_{lex}$. Then:

\begin{itemize}
\item The $143$ monomial ideals in $\mathcal{M}_{2t+2,3}$ define smooth \gr\ schemes. The dimensions are in the Table \ref{tab.decomp_2t+2_smooth}.
\item The following $16$ monomial ideals in $\mathcal{M}_{2t+2,3}$ define singular \gr\ schemes:
\[ \begin{aligned} &J_{1}, J_{2}, J_3, J_4, J_5, J_6, J_7, J_9, J_{10}, J_{11}, J_{12}, J_{14}, J_{15}\ \text{and}\\
&J_{16} = \langle w^3,zw^2,yw^2,yzw,y^2w,y^2z,y^3,xw^2,xzw,xyw, xy^2,x^2w \rangle,\\
&J_{17} = \langle z^2w, z^3, yz^2, xw^2, xzw, xz^2, xyw, xyz, x^2w, x^2z, x^2y, x^3 \rangle,\\
&J_{18} = \langle y^2w, y^2z, y^3, xw^2, xzw, xyw, xyz, xy^2, x^2w, x^2z, x^2y, x^3 \rangle.
\end{aligned} \]
\end{itemize}
The consequence is that $\mathrm{Hilb}^{2t+2}_{3}$ includes the $18$ singular points defined the above $18$ monomial ideals.\\

One may suppose that this method covers all singular points in $\hihilb$ defined monomial ideals by running monomial order $\prec$ enough. However, we do not have investigated it yet.
\end{example}

\begin{table}[h]
\begin{tabular}{|c||cccccccccccc|} \hline
$m$                & 0 & 1 & 2 & 3 & 4  & 5 & 6 & 7 & 8 & 9 & 10& 11\\ \hline \hline
$\prec_{rvlex}$ & 1 & 3 & 8 &18 & 23& 24&25&20 &14 & 6&2  & 0 \\ \hline
$\prec_{lex}$    &1 & 3 & 9 & 17& 22 & 24& 23& 19& 15& 6& 3& 1\\ \hline
\end{tabular}
\caption{The numbers of $J \in \mathcal{M}_{2t+2,3}$ such that $\hgrhilb \cong \mathbb{A}^m_k$ \label{tab.decomp_2t+2_smooth}}
\end{table}

\clearpage

\section{Homology formula}

We attach a $\gm$-action to $\prnn = \proj k[z_0, \ldots, z_N]$ such that $t \cdot [z_0,\ldots, z_N] = [t^{u_0} z_0, \ldots, t^{u_N}z_N]$. Let $X$ be a $k$-scheme with a $\gm$-equivariant embedding into $\prnn$, as like as the \plu\ embedding $\hihilb \hookrightarrow \mathbf{P}$. In this setting, \bb\ shows that the family of the BB schemes $\{X_i^{+} \}$ is filtrable \cite[Theorem 3]{BB76} (note that the proof only uses the existence of a $\gm$-equivariant embedding).  Here we say that $\{ X_i^{+} \}$ is \textit{filtrable} if there exists a sequence of closed subschemes $Z_0 \supset \cdots \supset Z_{q}$ such that
\begin{itemize}
\item $Z_0 = X$, $Z_q = \emptyset$,
\item each $Z_i \setminus Z_{i+1}$ is a BB scheme.
\end{itemize}
We obtain the followings by applying \cite[Theorem 3]{BB76} to the Hilbert scheme.

\begin{proposition}\label{prop.filtrable}
The family of the BB schemes $\{ \bbhilb \}$ in the Hilbert scheme $\hihilb$ is filtrable.
\end{proposition}

\begin{corollary}\label{cor.homology}
Assume that $k$ is algebraically closed, $\mathrm{char}\ k = 0$ and the Hilbert scheme $\hihilb$ is smooth. Denote by $p(J)$ the dimension of the Zariski tangent space of $\hgrhilb$ at $J$. Then we have the following formula about the homology of $\hihilb$: 
\[ H_m(\hihilb, \mathbb{Z}) \cong \bigoplus_{J \in \mpn} H_{m-2p(J)}(\{J\}, \mathbb{Z}) \cong \mathop{\bigoplus_{J \in \mpn}}_{2p(J) = m} \mathbb{Z} \]
for any integer $m$ with $0 \leq m \leq 2 \dim \hihilb$. 
\end{corollary}

\begin{proof}
From the hypothesis, $\bbhilb$ is isomorphic to the affine space $\mathbb{A}^{p(J)}_k$ for any $J \in \mpn$. Therefore we obtain the above formula by \cite[I\!I. Theorem 4.4, Corollary 4.15]{BBCM02}.
\end{proof}

\begin{example}\label{ex.2t+1}
Let $P = 2t +1$ and $n=3$. Then the Hilbert scheme $\mathrm{Hilb}^{2t+1}_{3}$ is smooth and the dimension is $8$. The numbers of monomial ideals $J \in \mathcal{M}_{2t+1,3}$ are in Table \ref{tab.betti_num_2t+1}. Therefore if $k$ is an  algebraically closed field with $\mathrm{char}\ k = 0$, the homologies $H_m = H_m(\mathrm{Hilb}^{2t+1}_{3}, \mathbb{Z})$ are the followings:
\[ H_0 = \mathbb{Z}, H_2 = \mathbb{Z}^2, H_4 = \mathbb{Z}^3, H_6 = \mathbb{Z}^4, \]
\[ H_8 = \mathbb{Z}^4, H_{10} = \mathbb{Z}^4, H_{12} = \mathbb{Z}^3, H_{14} = \mathbb{Z}^2, H_{16} = \mathbb{Z}. \]
\[ H_1 = H_3 = \cdots = H_{15} = 0. \]

\begin{table}[h]
\begin{tabular}{|c||ccccccccccc|} \hline
 $m$ & 0 & 1 & 2 & 3 & 4 & 5 & 6 & 7 & 8 & 9 & 10\\ \hline \hline
& 1 & 2 & 3 & 4& 4& 4&3 &2 &1 & & \\ \hline
\end{tabular}
\caption{The numbers of $J \in \mathcal{M}_{2t+1,3}$ such that $\hgrhilb \cong \mathbb{A}^m_k$ \label{tab.betti_num_2t+1}}
\end{table}

\end{example}

\bibliographystyle{alpha}
\bibliography{refloc}

\end{document}